\documentclass[a4paper]{amsart}

\usepackage{graphicx}
\usepackage{amsmath, amsthm, amsfonts, latexsym, amscd, amssymb, 
}
\newtheorem{thm}{Theorem}[section]
\newtheorem{cor}[thm]{Corollary}

\newtheorem{lem}[thm]{Lemma}
\newtheorem{prop}[thm]{Proposition}
\newtheorem{rem}[thm]{Remark}

\numberwithin{equation}{section}

\begin{document}
\title{On Commuting Graphs of Generalized Dihedral Groups}
\author{Vipul Kakkar and Gopal Singh Rawat}
\address{Department of Mathematics\\ Central University of Rajasthan\\ Ajmer, India}
\email{vplkakkar@gmail.com, gpl.chouhan.5@gmail.com}
\date{}

\begin{abstract}
For a group $G$ and a subset $X$ of $G$, the commuting graph of $X$, denoted by $\Gamma(G,X)$ is the graph whose vertex set is $X$ and any two vertices $u$ and $v$ in $X$ are adjacent if and only if they commute in $G$. In this article, certain properties of the commuting graph of generalized dihedral groups have been studied.
\end{abstract}
\maketitle

\section{Introduction}\footnotetext{This work is a part of M.Sc. Thesis of the second author.}
Throughout the article, we consider the simple undirected graphs which are without loops or multiple edges. For a graph $\Gamma$, the vertex set and edge set are denoted by $V( \Gamma )$ and $E( \Gamma )$. If a vertex $u$ is adjacent to a vertex $v$, then we denote it as $u \sim v$. The degree $deg(v)$ of a vertex $v$ in $\Gamma$ is the number of edges incident to $v$. A graph $\Gamma$ is said to be regular if and only if the degree of every vertex is equal. A graph $\Gamma$ is complete graph if each vertex are adjacent with every other vertex of the graph $\Gamma$ and denoted by $K_n$ where $n$ is the number of vertices in the graph.

\vspace{0.2 cm}

A \textit{generalized dihedral group} $D(G)$ is the semi-direct product $G \rtimes_\phi C_2$ of an abelian group $G$ with a cyclic group $C_2=\{1,-1\}$, where homomorphism $\phi$ maps $1$ and $-1$ to identity automorphism and inversion automorphism respectively. Therefore, the binary operation on $D(G)$ is defined as follows

\[
(g_1,\, c_1)\,(g_2,\,c_2)\,=\,(g_1\,{g_2}^{c_1}\, , \, c_1\,c_2\,),
\]
where $g\in G$ and $c \in C_2$.

\vspace{0.2 cm}

In this paper, $G$ will denote a finite abelian group of order $n$. Then, $G$ is isomorphic to the direct product $\mathbb{Z}_{m_{1}} \times \cdots \times \mathbb{Z}_{m_{k}} $ of cyclic groups, where $m_{1} \cdots m_{k}=n$. In this paper, we will identify $G$ with $\mathbb{Z}_{m_{1}} \times \cdots \times \mathbb{Z}_{m_{k}}$ and we place factors of direct product of $2$-power order (if it exists) before the factors of odd order, that is \[G = \underbrace{\mathbb{Z}_{m_{1}} \times \cdots \times \mathbb{Z}_{m_{r}}}_{\text{factors of $2$-power order}} \times \overbrace{ \mathbb{Z}_{m_{r+1}} \times \cdots \times \mathbb{Z}_{m_{k}}}^{\text{factors of odd order}}\]

\vspace{0.2 cm}

In the above, $m_i=2^t$ for some integer $t \geq 0$ and $1\leq i \leq r$. By \cite[Theorem 2.9, p. 9]{bb}, we note that $D(G)$ is abelian if and only if $G$  is elementary abelian $2$-group. Throughout the paper, by a generalized dihedral group $D(G)$ we will always mean it to be non-abelian, otherwise it will be stated. Following are elementary results.

\begin{lem}\label{s2l1} (\cite[Proposition 4.12, p.11]{bb})
Given any abelian group $G$, the element $(g,1)$ is in $Z(D(G))$ if and only if $g^2=e \in G$.
\end{lem}    
 
\begin{lem}\label{s2l2} (\cite[Proposition 4.13, p.12]{bb})
Given any abelian group $G$, the element $(g, -1)$ is in the center of $D(G)$ if and only if $D(G)$ is abelian.
\end{lem}
 
\begin{lem}\label{s2l3} 
Let $G$ be an abelian group. Then $(g_1, -1),(g_2,-1)\in D(G)$ commute if and only if $g_1^2=g_2^2$.
\end{lem}

\begin{lem}\label{s2l4} 
Let $G$ be an abelian group. If $(g_1, 1)\in D(G)$ commute with $(g_2,-1)\in D(G)$, then $(g_1,1)\in Z(D(G))$.
\end{lem}

By Lemmas \ref{s2l1} and \ref{s2l2}, one observes that $Z(D(G)) = \{(g, 1) \,|\, g^2 = e\}$. This implies that $|Z(D(G))|=2^r$. Let $G_1=\{(g,1)\in D(G) \mid  g \in G\}$ and $G_2=\{(g,-1) \in D(G) \mid g \in G\}$. Note that $Z(D(G)) \subseteq G_1$. Further, we write $\Omega_1 = Z(D(G))$, $\Omega_2=G_1 \setminus Z(D(G))$ and $\Omega_3 = G_2$. Let $g_1,g_2 \in G=\mathbb{Z}_{m_{1}} \times  \cdots \mathbb{Z}_{m_{r}} \times  \mathbb{Z}_{m_{r+1}} \times \cdots \times \mathbb{Z}_{m_{k}}$. Then $g_i=(g_{i1},\cdots,g_{ik})$, $i=1,2$. Let $(g_1,-1)$ commutes with $(g_2,-1)$. Then, by Lemma \ref{s2l3}, $g_{1j}^2=g_{2j}^2$ for all $1 \leq j \leq k$. This implies that the number of elements in $G_2$ that commute with a fix element $(g_1,-1)$ is $2^r$. Also, one notes that $g_1^2=g_2^2$ defines an equivalence relation in $G$. Therefore, $G_2$ is partitioned into $\frac{n}{2^r}$ subsets. We denote these subsets by $B_1,\cdots, B_{\frac{n}{2^r}}$.   

\vspace{0.2 cm}

In the last twenty years, plenty of researchers have been attracted to study the graphs of algebraic structure. The study of algebraic structures, using the properties of graphs, has become an exciting research topic in these years, leading to many fascinating results and raising questions. The commuting graph of a group is studied by various author (see \cite{ab}-\cite{ali} and \cite{che}-\cite{seg}). In \cite{ali}, the certain properties of the commuting graph of dihedral group are studied. In this paper, we have studied those properties for the commuting graph of generalized dihedral group

\section{Commuting Graph of $D(G)$}

For a non-empty subset $X$ of $D(G)$, the commuting graph of $X$ denoted by $\Gamma (X) = \mathfrak {C}(D(G),X) $ is a graph whose vertex set $V(\Gamma)$ is $X$ and any two vertices $u$ and $v$ are adjacent $(u \sim v)$ if and only if $uv=vu$. If $m_i =2 \, \forall \, i \in \{1,2. \cdots,k\}$ then $D(G)$ is abelian hence $ \mathfrak {C}(D(G),D(G))=K_{2^n}$. Since $\Omega_2 \subseteq G_1$, each element of $\Omega_2$ commutes with each other. Also, no element of $\Omega_2$ commute with any element of $G_2$. Therefore, we obtain the following.

\begin{prop}\label{s3p1} If $X$ is subset of $D(G)$, then   
\[\mathfrak {C}(D(G),X)= \begin{cases}
  K_{2^r} & \text{if } \, X=\Omega_1\\
  K_{n-2^{r}} & \text{if } \,X=\Omega_2\\
  K_{2^r} & \text{if } X=B_i~(1 \leq i \leq \frac{n}{2^p})
	\end{cases}
  \]
	\end{prop}

The join $ \Gamma ^{\prime} \vee \Gamma^{\prime \prime}$ of two graphs $ \Gamma ^{\prime}$ and $ \Gamma ^{\prime \prime}$ is a graph with vertex set $V(\Gamma ^{\prime}) \cup V(\Gamma ^{\prime \prime})$ and an edge set $E( \Gamma ^{\prime}) \cup E( \Gamma ^{\prime \prime}) \cup \{u \sim v \, | \, u \in V( \Gamma^{\prime}) \, \& \, v \in \Gamma^({\prime \prime})\}$ and the union $\Gamma ^{\prime} \cup \Gamma^{\prime \prime}$ of two graphs $ \Gamma ^{\prime}$ and $ \Gamma^{\prime \prime}$
 is a graph with vertex set $V(\Gamma ^{\prime}) \cup V(\Gamma ^{\prime \prime})$ and an edge set $E(\ \Gamma^{\prime}) \cup E( \Gamma ^{\prime \prime})$. If a graph $\Gamma$ is union of $m$ complete graph $K_n$ with $n$ vertices, then we write $\Gamma = mK_n$. 

\begin{cor}\label{s3p1c1}
Let $\Gamma= \mathfrak {C}(D(G),D(G))$ be the commuting graph of $D(G)$. Then $\Gamma=K_{2^r} \vee (K_{n-2^r} \cup \frac {n}{2^r}K_{2^r}) $.
\end{cor}
\begin{proof}
By Proposition \ref{s3p1}, there is no edge between two vertices from distinct blocks $B_i$. Therefore,  $\mathfrak {C}(D(G),\Omega_3)= \bigcup \limits _{i=1}^{n/2^r} \mathfrak {C}(D(G),B_{i})= \frac {n}{2^r}K_{2^r}$. Also note that there is no edge between vertices from $\Omega_2$ and $\Omega_3$. Further, since $\Omega_1$ is center of $D(G)$, each vertex from $ \Omega_2 \cup \Omega_3$ is adjacent with each vertex of $\Omega_1$. Hence $ \mathfrak {C}(D(G),D(G)) = K_{2^r} \vee \{ K_{n-2^r} \cup \frac {n}{2^r}K_{2^r} \} $.
\end{proof}

In a simple undirected graph $\Gamma$, degree of a vertex $v$ is the number of edges incident to that vertex and denoted by $deg(v)$. Now, we have the following.
\begin{cor}\label{s3p1c2}
Let $\Gamma= \mathfrak {C}(D(G),D(G))$ be the commuting graph of $D(G)$ and $v$ be a vertex of $\Gamma$. Then
\[
deg(v)= \begin{cases}

2n-1 & \text{if} \, v \in \Omega_1 \\ 
n-1  & \text{if} \, v \in\Omega_2\\ 
2^{r+1}-1 & \text{if}  \, v \in \Omega_3 
\end{cases}
\]
\end{cor}

\begin{cor}\label{s3p1c3}
Let $\Gamma= \mathfrak {C}(D(G),D(G))$ be the commuting graph of $D(G)$ and $v$ be a vertex of $\Gamma$. Then\\
	\[
	|E(\Gamma)|= 3n2^{r-1}+ \frac{n(n-2)}{2}
	\]	
\end{cor}
The coloring of a graph is an assignment of colors to the vertices of the graph so that no two adjacent vertices have the same color. The chromatic number of a graph is the smallest number of colors needed to color the graph and denoted by $\psi (\Gamma)$. We have the following proposition.
\begin{prop}\label{s33p2}
Let $\Gamma= \mathfrak {C}(D(G),D(G))$ be the commuting graph of $D(G)$. Then $\psi (\Gamma)= n$.
\end{prop} 
\begin{proof}  We start coloring from $\Omega_1 \cup \Omega_2=G_1$. Since $ \mathfrak {C}(D(G),\Omega_1 \cup \Omega_2)$ is complete graph of $n$ vertices, we need $n$ colors to get it colored. Therefore $(\Omega_1 \cup \Omega_2)$ is colored with $n$ colors. To color vertices from $\Omega_3$, we move block wise. In each block there are $2^r$ vertices and these all are adjacent with vertices from $\Omega_1$ but not adjacent with vertices from $\Omega_2$. Since $|\Omega_2|>2^r$, we can choose any of $2^r$ colors that we used for $\Omega_2$. These colors again can be used to color another block as any block has no common edge with the rest of blocks. Hence, $\psi(\Gamma)=n$.
\end{proof}

\section{Detour Distance of Commuting graph of $D(G)$}

The detour distance $d_D(u,v)$ between two vertices $u$ and $v$ in a graph $\Gamma$ is the length of a longest $u-v$ path in $\Gamma$. A $u-v$ path of length $d_D(u,v)$ is called a $u-v$ detour geodesic. The detour eccentricity $(ecc_D(v))$ of a vertex $v$ in $\Gamma$ is the maximum detour distance between $v$ and any vertex of $\Gamma$. The minimum detour eccentricity among the vertices of $\Gamma$ is called the detour radius of $\Gamma$, denoted by $rad_D(\Gamma)$. The detour diameter $diam_D(\Gamma)$ of a graph $\Gamma$ is the maximum detour eccentricity in $\Gamma$. Now, we have the following.

 
\begin{thm}\label{s3t1}
Let  $\Gamma= \mathfrak {C}(D(G),D(G))$ be the commuting graph of $D(G)$ with $|\Omega_1|=2^r$. Then for each $v \in \Omega_1$
\[
ecc_D(v) =
\begin{cases}
2n-1               & \text{if} \, \frac{n}{2^r}<2^r \\
n+2^{r}(2^{r}-1)-1 & \text{if} \,\frac{n}{2^r} \geq 2^{r} \\
\end{cases}
\]
and for each $v \in \Omega_2 \cup \Omega_3$
\[
ecc_D(v) =
\begin{cases}
2n-1             & \text{if} \, \frac{n}{2^r} \leq 2^r \\
n+4^{r}-1        & \text{if} \, \frac{n}{2^r} > 2^{r} \\
\end{cases}
\]
\end{thm} 
\begin{proof}
First assume that $v \in \Omega_1$. If $\frac{n}{2^r}<2^r$, then the center $\Omega_1$ has more elements than number of blocks. To get a maximum length path starting from $v \in \Omega_1$, one first covers all vertices of $\Omega_2$, for they all are adjacent. Therefore, one covers each vertex from $\Omega_2$ without repetition. Now, one moves to any vertex of $\Omega_1$ except $v$ itself and from there one moves to one of the blocks and complete all the vertices of that block. Now, move to another vertices in $\Omega_1$ from there we again cover a block. Repeating this process we cover all the blocks. If there are some vertices still left in $\Omega_1$, then they will be covered at last. Hence $ecc_D(v)=2n-1$. If $\frac{n}{2^r} > 2^{r}$, then we do the same process as above but the vertices from $\Omega_1$ exhausts before it cover all the blocks. Hence in this case the maximum blocks that can be covered is $2^r-1$. Therefore, the total vertices which are covered  are $n+2^{r}(2^{r}-1)$ by the longest path started from a vertex $v \in \Omega_1$. Hence, $ecc_D(v)=n+2^{r}(2^{r}-1)-1$.

\vspace{0.2 cm}

Now, assume that $v \in \Omega_2$ and $\frac{n}{2^r} \leq 2^r$. We proceed in the same manner as we did in the above case. We start from vertex $v \in \Omega_2$ and first cover all the vertices of $\Omega_2$ and then move to one of the vertex of $\Omega_1$. Then we head to one of the block and cover it and again move to another vertices of $\Omega_1$. Keep on repeating the process we first exhaust with blocks and still left some vertices in $\Omega_1$ which we cover at the end. Hence, $ecc_D(v)=2n-1$.  If $\frac{n}{2^r} > 2^{r}$, then one can similarly observed that $ecc_D(v)=n+4^r-1$.

\vspace{0.2 cm}

Finally, assume that $v \in \Omega_3$ and $\frac{n}{2^r} \leq 2^r$. We start from one vertex from a block and cover it and then moves to one of the vertex in $\Omega_1$ and then covers whole $\Omega_2$. Then , as above $ecc_D(v)=2n-1$. If $\frac{n}{2^r} > 2^{r}$, then by the similar argument as above $ecc_D(v)=n+4^{r}-1$.
\end{proof} 

\begin{cor}
Let  $\Gamma= \mathfrak {C}(D(G),D(G))$ be the commuting graph of $D(G)$ with $|\Omega_1|=2^r$. Then
\[
rad_D(\Gamma) =
\begin{cases}
2n-1               & \text{if $\frac{n}{2^r} < 2^r$} \\
n+2^{r}(2^{r}-1)-1 & \text{if $\frac{n}{2^r} \geq 2^{r}$} \\
\end{cases}
\]
and
\[
diam_D(\Gamma) =
\begin{cases}
2n-1             & \text{if $\frac{n}{2^r} \leq 2^r$} \\
n+4^{r}-1        & \text{if $\frac{n}{2^r} > 2^{r}$} \\
\end{cases}
\]
\end{cor}

\section{Resolving Polynomial of Commuting Graph of $D(G)$}
Let $\beta(G)$ the metric dimension of $\Gamma$ and $\beta (\Gamma,x)= \sum_{i=\beta (\Gamma)}^{n}s_ix^{i}$ denote the resolving polynomial of $\Gamma$ (see \cite{ali}).


Let $u$ be a vertex of a graph $\Gamma$. Then, the set $N(u)=\{v \in V(\Gamma) \, | \, v \sim u \text{ in } \Gamma \}$ is called the open neighborhood of $u$ and $N[u]=N(u) \cup \{u\}$ is called the closed neighborhood of $u$. Two distinct vertices $u$ and $v$ of $\Gamma$ are called twins if $N[u]=N[v]$ or $N(u)=N(v)$. A subset $U$ of vertex set of $\Gamma$ is called a twin-set in $\Gamma$
if $u,v$ are twins in $\Gamma$ for every pair of distinct vertices $u,v \in U$. The following the remark from \cite{ali}.

\begin{rem}(\cite[Remark 3.3, p. 2398]{ali})\label{s4r1}
If $U$ is a twin-set in a connected graph $\Gamma$ of order $n$ with $|U|=l \geq 2$,
then every resolving set for $\Gamma$ contains at least $l-1$ vertices of $U$.
\end{rem}

\begin{thm}\label{s4t1}
Let $\Gamma= \mathfrak {C}(D(G),D(G))$ be the commuting graph of $D(G)$.Then 
\begin{equation*}
\beta (\Gamma)= \begin{cases} 
2n-\frac{n}{2^{r}}-2 & |\Omega_1| \geq 2\\
2n-3                 & |\Omega_1|=1\\
\end{cases}
\end{equation*}
\end{thm}
\begin{proof}
Assume that $|\Omega_1|=2^r \geq 2$. Note that $N[u]=N[v]=D(G) \, \, \forall \, u,v \in \Omega_1$ and there does not exist  $w \in D(G)  \backslash \Omega_1$ such that $N[w]=D(G)$. Hence $\Omega_1$ is a twin-set in $\Gamma$ with $|\Omega_1|\geq 2$. Further $N[u]=N[v]= \Omega_1 \cup \Omega_2 \, \, \forall \, u,v \in \Omega_2$ and there does not exist  $w \in D(G)  \backslash \Omega_2$ such that $N[w]= \Omega_1 \cup \Omega_2$. Hence $\Omega_2$ is a twin-set in $\Gamma$ with $|\Omega_2| \geq 2$ as $|\Omega_2| \geq |\Omega_1|$. Similarly each block $B_i(1 \leq i \leq \frac{n}{2^r})$ is a twin-set as $N[u]=N[v]=\Omega_1 \cup B_i \, \, \forall \, u,v \in B_i(1 \leq i \leq \frac{n}{2^r})$ with $|B_i| \geq 2$. Hence all twin-sets in $D(G)$ with cardinality greater than or equal to $2$ are $\Omega_1 , \Omega_2, B_1, B,_2, \cdots B_{\frac{n}{2^r}}$. Therefore, by Remark \ref{s4r1}, $\beta (\Gamma)=2n-\frac{n}{2^r}-2$.

\vspace{0.2 cm}

Now assume that $\Omega_1=1$. One can note that $\Omega_2$ is a twin-set with $|\Omega_2| \geq 2$ and $\Omega_3$ is also a twin-set as $N(u)=N(v)=\Omega_1 \, \forall \, u,v \in \, \Omega_3$. Therefore, there are two twin-sets in this case. Hence $\beta (\Gamma)= |\Omega_2|-1 +|\Omega_3|-1 = (n-1)-1+n-1=2n-3$.
\end{proof}

Now, we find the resolving polynomial $\beta(\Gamma,x)$ of the graph $\Gamma$.

\begin{thm}\label{s4t2}
Let $\Gamma= \mathfrak {C}(D(G),D(G))$ be the commuting graph with $|\Omega_1|=1$. Then
 \begin{equation*}
 \beta (\Gamma,x)= x^{2n-3}(x^3 + 2nx^2 +(n^2+n-1)x+n(n-1))
 \end{equation*}
\end{thm} 
\begin{proof}
By Theorem \ref{s4t1}, $\beta (\Gamma,x)=2n-3$. In order to find the resolving polynomial, we need to calculate $s_{2n-3},s_{2n-2},s_{2n-1},s_{2n}$. By \cite[Proposition 3.5, p. 2399]{ali}, $s_{2n-1}=2n$ and $s_{2n}=1$. 

\vspace{0.2 cm}

The sets $\Omega_1, \Omega_2$ and $\Omega_3$ are mutually disjoint out of which $\Omega_2$ and $\Omega_3$ are twins set. There for we have to pick $|\Omega_2|-1$ vertices from $\Omega_2$ and $|\Omega_3|-1$ vertices from $\Omega_3$ and no vertex from $\Omega_1$. Therefore, the total number of choices of resolving set for $\Gamma$ of cardinality $2n-3$ is
\begin{equation*}
s_{2n-3}={^1C_0} \times {^{n-1}C_{n-2}} \times {^{n}C_{n-1}}=(n-1)n=n(n-1)
\end{equation*}

Now, we calculate ${s_{2n-2}}$. In this case, we have to choose one more vertices than $2n-3$. This may be from $\Omega_1,\Omega_2$ or $\Omega_3$. Therefore, the total number of choices of resolving sets of cardinality $2n-2$ is
\begin{equation*}
s_{2n-2}={^1C_1} \times {^{n-1}C_{n-2}} \times {^{n}C_{n-1}}+{^1C_0} \times {^{n-1}C_{n-1}} \times {^{n}C_{n-1}}+{^1C_0} \times {^{n-1}C_{n-2}} \times {^{n}C_{n}}
\end{equation*}
\begin{equation*}
\implies s_{2n-2}=(n-1)n+n+n-1=n^2+n-1
\end{equation*}
\end{proof}
 
\begin{thm}\label{s4t3}
Let $\Gamma= \mathfrak {C}(D(G),D(G))$ be the commuting graph with $|\Omega_1| \geq 2 $. Then
\[
\beta(\Gamma,x)=x^{2n- \frac{n}{2^{r}}-2} \Bigg((n-2^r)(2^r)^{\frac{n}{2^{r}}+1} + \sum_{i=2n-\frac{n}{2^r}-1}^{2n-2}s_{i}x^{i}+2nx^{\frac{n}{2^r}+1}+x^{\frac{n}{2^{r}}+2}\Bigg)
\]
where\\
\begin{equation*}
s_i=(n-2^r)(2^r)^{2n-i-1} \times {^{\frac{n}{2^r}+1}C_{2n-i-1}} + (2^r)^{2n-i} \times {^{\frac{n}{2^r}+1}C_{2n-i}}
\end{equation*}
\end{thm}
\begin{proof} We will calculate $s_{j}$, where $j=2n-\frac{n}{2^r}-2$ and $s_i \, \, (2n-\frac{n}{2^r}-1 \leq i \leq 2n-2)$.

\vspace{0.2 cm}
 
We first calculate $s_{j}$, $j=2n-\frac{n}{2^r}-2$. Since $|\Omega_1| \geq 2$, all the twin-sets in $D(G)$ with cardinality greater than or equal to $2$ are $\Omega_2 , \Omega_1, B_1, B,_2, \cdots B_{\frac{n}{2^r}}$. We have to choose all vertices except any one vertex from each of twin-sets. Therefore, the total ways to form such a resolving set is
\begin{equation*}
s_j={^{n-2^r}C_{n-2^{r}-1}} \times {^{2^r}C_{2^r-1}} \times \underbrace {{^{2^r}C_{2^r-1}} \times \cdots {^{2^r}C_{2^r-1}}}_{\frac{n}{2^r}-\text{times}} 
\end{equation*}
\[
s_j=(n-2^r)\times 2^r \times \underbrace { 2^r \cdots \times 2^r}_{\frac{n}{2^r}-\text{times}}=(n-2^{r}) \bigg(2^r\bigg)^{\frac{n}{2^{r}}+1}
\]

Now, we calculate $s_i \, \, \,(2n-\frac{n}{2^r}-1 \leq i \leq 2n-2)$. Let us choose a resolving set of cardinality $2n-\frac{n}{2^r}-2+t$ with $1 \leq t \leq \frac{n}{2^r}$, that is, we choose $t$ more vertices than $\beta(\Gamma)=2n-\frac{n}{2^r}-2$. Observe that except $\Omega_2$, all the twin-sets $\Omega_1, B_1, B,_2, \cdots B_{\frac{n}{2^r}}$ have cardinality $2^r$. We have to choose $t$  more vertices. First, choose one of $t$ vertices from $\Omega_2$ and the rest from $\bigg(\bigcup_{j=1}^{\frac{n}{2^r}}B_j \bigg) \cup \Omega_1$ and then choose all from $\bigg(\bigcup_{j=1}^{\frac{n}{2^r}}B_j \bigg) \cup \Omega_1$. Therefore, the total choices are
\[
s_{2n-\frac{n}{2^r}-2+t}= {^{n-2^r}C_{n-2^r}} \times \Bigg ( \bigg( \underbrace {{^{2^r}C_{2^r}} \times \cdots \times {^{2^r}C_{2^r}}}_{(t-1) \text{times}}    \bigg)
\]
\begin{flushright}
$\times \bigg ( \underbrace {{^{2^{r}}C_{2^{r}-1}} \times \cdots \times {^{2^{r}}C_{2^{r}-1}}}_{\frac{n}{2^{r}}-t + 2} \bigg)
\times \bigg ( {^{\frac{n}{2^{r}}+1}C_{t-1}}\bigg )   \Bigg )$ 
\end{flushright}
\begin{flushright} $ + \quad           {^{n-2^r}C_{n-2^r-1}} \times \Bigg ( \bigg( \underbrace {{^{2^r}C_{2^r}} \times \cdots \times {^{2^r}C_{2^r}}}_{t  \, \, \text{times}}    \bigg)$
\end{flushright}
\begin{flushright}
$\times 
\bigg( \underbrace {{^{2^r}C_{2^r-1}} \times \cdots \times {^{2^r}C_{2^r-1}}}_{(\frac{n}{2^{r}}-t+1) \, \text{times}}    \bigg)   \times \bigg ( {^{\frac{n}{2^{r}}+1}C_{t}}\bigg )   \Bigg )$
\end{flushright}
\[
s_{2n-\frac{n}{2^r}-2+t}={(2^{r})}^{\frac{n}{2^{r}}-t+2} \times {^{\frac{n}{2^r}+1}C_{t-1}} + (n-{2^r}) \times {{(2^r)}^{\frac{n}{2^r}-t+1}} \times {^{\frac{n}{2^r}+1}C_{t}}
\]
Suppose $2n-\frac{n}{2^{r}}-2+t=i$. This implies $t=i-2n+\frac{n}{2^{r}}+2$ and $2n-\frac{n}{2^r}-1 \leq i \leq 2n-2$. Therefore, the total ways of choosing resolving set of cardinality $i$ is

\[
s_i={(2^{r})}^{\frac{n}{2^{r}}-i+2n-\frac{n}{2^{r}}-2+2} \times {^{\frac{n}{2^{r}}+1}C_{i-2n+\frac{n}{2^{r}}+2-1}} + (n-2^{r}) \times {(2^{r})}^{\frac{n}{2^{r}}-i+2n-\frac{n}{2^{r}}-2+1}
\]

\begin{flushright}		
$\times {^{\frac{n}{2^{r}}+1}C_{i-2n+\frac{n}{2^{r}}+2}}$
\end{flushright}

\[
s_i={(2^r)}^{2n-i} \times {^{\frac{n}{2^{r}}+1}C_{i-2n+\frac{n}{2^{r}}+1}} + (n-2^{r}) \times {(2^{r})}^{2n-i-1} \times  {^{\frac{n}{2^{r}}+1}C_{i-2n+\frac{n}{2^{r}}+2}}
\]

\begin{flushleft}
Since ${^{n}C_r}={^{n}C_{n-r}}$,
\end{flushleft}
\[
s_i={(2^r)}^{2n-i} \times {^{\frac{n}{2^{r}}+1}C_{2n-i}} + (n-2^{r}) \times {(2^{r})}^{2n-i-1} \times  {^{\frac{n}{2^{r}}+1}C_{2n-i-1}}.
\]
\end{proof}

\end{document}